\documentclass{amsart}
\usepackage{amssymb,amscd}

 \newtheorem{thm}{Theorem}[section]
 \newtheorem{cor}[thm]{Corollary}
 \newtheorem{lemma}[thm]{Lemma}
 \newtheorem{prop}[thm]{Proposition}
 \theoremstyle{definition}
 \newtheorem{defn}[thm]{Definition}
 \theoremstyle{remark}
 \newtheorem{rem}[thm]{Remark}
 \numberwithin{equation}{subsection}
\newcommand{\TT}{\text{$\mathcal{T}$}}
\newcommand{\BB}{\text{$\mathcal{B}$}}
\newcommand{\MM}{\text{$\mathcal{M}$}}

\newcommand{\LB}{\text{$\Lambda$}}
\newcommand{\sg}{\text{$\sigma$}}

        \newcommand{\field}[1]{\text{$\mathbb{#1}$}}
        \newcommand{\N}{\field{N}}
        \newcommand{\Z}{\field{Z}}
        
        \newcommand{\R}{\field{R}}

\begin{document}

\title
 {Transverse invariant measures extend to the ambient space}

\author{Carlos Meni\~no Cot\'on}

\address{Departamento de Xeometr\'{\i}a e Topolox\'{\i}a\\
         Facultade de Matem\'aticas\\
         Universidade de Santiago de Compostela\\
         15782 Santiago de Compostela}

\email{carlosmecon@hotmail.com}

%%% ----------------------------------------------------------------------

\begin{abstract}
It is shown that any transverse invariant measure of a foliated space can be considered as a measure on the ambient
space.
\end{abstract}

%%% ----------------------------------------------------------------------
\maketitle
%%% ----------------------------------------------------------------------

\section{Introduction}
Transverse invariant measures of foliated spaces play an important
role in the study of their transverse dynamics. They are measures
on transversals invariant by ho\-lo\-no\-my transformations. There
are many interpretations of transverse invariant measures; in
particular, they can be extended to generalized transversals,
which are defined as Borel sets that meet each leaf in a countable
set \cite{Connes}. Here, we show that indeed invariant measures
can be extended to the \sg-algebra of all Borel sets becoming an
``ambient'' measure (a measure on the ambient space). Precisely,
the following result is proved.

\begin{thm}\label{t:extension medida}
Let $X$ be a foliated space with a transverse invariant measure
\LB. There exists a Borel measure $\widetilde{\LB}$ on $X$ such
that $\widetilde{\LB}(T)=\LB(T)$ for all generalized transversal
$T$.
\end{thm}

This $\widetilde{\LB}$ is constructed as a ``pairing'' of the
transverse invariant measure $\LB$ with the counting measure on
the leaves. Connes has proved that this pairing is coherent for
ge\-ne\-ra\-lized transversals but it can not be directly extended
to any Borel set since the local projection of a Borel set is not
necessarily a Borel set \cite{Kechris,Srivastava}. The solution of
this problem is the main difficulty of the proof. In fact, this
result is proved in the more general setting of foliated
measurable spaces \cite{Bermudez,Bermudez-Hector}. The uniqueness of this
extension is also discussed.

The extension $\widetilde{\LB}$ will be used in study of the
concept of {\em \LB-LS category\/}, which will appear elsewhere.

\section{Foliated measurable spaces}

A {\em Polish space \/} is a
completely metrizable and separable topological space. A {\em
standard Borel space\/} is a measurable space isomorphic to a
Borel subset of a Polish space. A {\em measurable topological
space\/} or {\em MT-space\/} $X$ is a set equipped with a
\sg-algebra and a topology. Usually, measure theoretic concepts
will refer to the \sg-algebra of $X$, and topological concepts
will refer to its topology. Notice that the \sg-algebra does not
necessarily agree with the Borel \sg-algebra associated with the
topology. An {\em MT-map\/} between MT-spaces is a measurable
continuous map. An {\em MT-isomorphism} is a map which is a
measurable isomorphism and a homeomorphism, simultaneously.

Let $T$ be a standard Borel space. On $T\times\R^n$, we consider
the \sg-algebra generated by products of Borel subsets of $T$ and
$\R^n$, and the product of the discrete topology on $T$ and the
Euclidean topology on $\R^n$. $T\times\R^n$ will be endowed with
the structure of MT-space defined by this \sg-algebra and this
topology.

A {\em foliated measurable chart} on $X$ is an MT-isomorphism
$\varphi:U\to T\times\R^n$, where $U$ is open and measurable in
$X$. A {\em foliated measurable atlas\/} on $X$ is a countable
family of foliated measurable charts whose domains cover $X$. The
sets $\varphi^{-1}(\{\ast\}\times\R^n)$ are the {\em plaques\/} of
the foliated chart $\varphi$ and the sets
$\varphi^{-1}(T\times\{\ast\})$ are called {\em transversals\/}
induced by $\varphi$. A {\em foliated measurable space\/} is an
MT-space that admits a foliated measurable atlas. Observe that we
always consider countable atlases. The connected components of $X$
are called its {\em leaves\/}. An example of foliated measurable
space is a foliated space with its Borel \sg-algebra and the leaf
topology. According to this definition, the leaves are second
countable connected manifolds but they may not be Hausdorff.

A measurable subset $T\subset X$  is called a {\em generalized
transversal\/} if its intersection with each leaf is countable;
these are slightly more general than the transversals of
\cite{Bermudez}, which are required to have discrete and closed
intersection with each leaf. Let $\TT(X)$ be the set of
generalized transversals of $X$. This set is closed under
countable unions and intersections, but it is not a \sg-algebra.

A {\em measurable holonomy transformation\/} is a measurable
isomorphism $\gamma:T\to~T'$, for $T, T' \in \TT(X)$, which maps
each point to a point in the same leaf. A {\em transverse
invariant measure\/} on $X$ is a \sg-additive map $\LB:\TT(X)\to
[0,\infty]$ which is invariant by measurable holonomy
transformations. The classical definition of transverse invariant
measure in the context of foliated spaces is a measure on
usual transversals invariant by usual
holonomy transformations \cite{CandelConlon}. Both definitions are
equivalent in the case of foliated measurable spaces induced by
foliated spaces \cite{Connes}.

\section{Case of a product foliated measurable
space}\label{s:producto}

In this section, we take foliated measurable spaces of the form
$T\times P$, where $T$ is a standard measurable space and $P$ a
connected, separable and Hausdorff manifold. Indeed, the results
of this section hold when $P$ is any Polish space. We assume that
a new topology is given in this space as follows. All standard
Borel spaces are Borel isomorphic to a finite set, \Z\ or the
interval $[0, 1]$ (see \cite{Srivastava,Takesaki}). Identify
$T\times P$ with $\Z\times P$, $[0, 1]\times P$ or $A\times P$
($A$ finite), via a Borel isomorphism. We work with two topologies in
$T\times P$. On the one hand, the topology of the MT-structure is the product of
discrete topology on $T$ and the topology of $P$. On the other hand, the
topology is the product of the topology of $[0,1]$, $\N$ or $P$ with the topology of $P$; the term ``open set'' is used
with this topology. The \sg-algebra of the
MT-structure on $T\times P$ is generated by these ``open sets''.
Let $\pi: T\times P \to T$ be the first factor projection.

\begin{prop}[R.Kallman \cite{Kallman}]\label{p:sgcompactos}
If $B\subset T\times P$ is a Borel set such that $B\cap
(\{t\}\times P)$ is \sg-compact for all $t\in T$, then $\pi(B)$ is
a Borel set. Moreover there exists a Borel subset $B'\subset B$
such that $\#(B'\cap(\{t\}\times P))=1$ if $B\cap (\{t\}\times
P)\neq\emptyset$, and $\#(B'\cap (\{t\}\times P))=0$ otherwise.
\end{prop}

For any measurable space $(X,\MM,\LB)$, the
{\em completion\/} of \MM\ with respect to \LB\ is the \sg-algebra
$$\MM_\LB=\{\,Z\subset X\ |\ \exists\, A,B\in \MM,\ A\subset
Z\subset B,\ \LB(B\setminus A)=0\,\}\;.$$ The measure $\LB$ extends
in a natural way to $\MM_\LB$ by defining $\LB(Z)=\LB(A)=\LB(B)$ for
$Z,A,B$ as above.

Now, let \LB\ be a Borel measure on $T$. Define
$$\pi(\BB^*,\BB_\LB)=\{\,B\subset \BB^*\ |\ \pi(B\cap U)\in \BB_\LB\
\forall\ \text{open}\ U\subset T\times P\,\}\;,$$where $\BB$ and
$\BB^*$ are the Borel \sg-algebras of $T$ and $T\times P$,
respectively.

\begin{prop}\label{p:propiedadunion}
$\pi(\BB^*,\BB_\LB)$ is closed under countable unions.
\end{prop}
\begin{proof}
This is obvious since, for any countable family $\{B_n\}\subset
\BB^*$, we obtain $\bigcup_n B_n\in\BB^*$ and $$\pi\left(\left(\bigcup_n
B_n\right)\cap U\right)=\bigcup_n\pi(B_n\cap U)\in\BB_\LB\;$$ for any open
subset $U\subset T\times P$.
\end{proof}

\begin{rem}\label{r:analytic}
If $\LB$ is \sg-finite ({\em i.e.\/}, $T$ is a countable union of
Borel sets with finite \LB-measure), then
$\pi(\BB^*,\BB_{\LB})=\BB^*$\,: by Exercise 14.6 in
\cite{Kechris}, any set in $\BB^*$ projects onto an analytic set,
which is \LB-measurable since $\LB$ is \sg-finite \cite[Theorem
4.3.1]{Srivastava}.
\end{rem}

\begin{rem}
If $B\in \pi(\BB^*,\BB_\LB)$ and $U$ is an open set, then $B\cap
U\in \pi(\BB^*,\BB_\LB)$. By Proposition \ref{p:sgcompactos},
$\pi(\BB^*,\BB_\LB)$ contains the Borel sets with \sg-compact
intersection with the plaques $\{t\}\times P$.
\end{rem}

Now, we want to
extend \LB\ to all Borel sets satisfying the conditions of a
measure. Let $\BB^{**}$ denote the Borel \sg-algebra of $T\times P\times
T\times P$, $\widetilde{\pi}$ the natural projection $T\times
P\times T\times P\to T\times T$, and
$\langle\BB_\LB\times\BB_\LB\rangle$ the \sg-algebra generated by
sets of the form $A\times B$ for $A,B\in\BB_\LB$.
\begin{lemma}\label{l:productos}
If $B,B'\in \pi(\BB^*,\BB_\LB)$, then $B\times B'\in
\widetilde{\pi}(\BB^{**},\langle\BB_\LB\times\BB_\LB\rangle)$.
\end{lemma}
\begin{proof}
Since $B,\ B'\in \BB^*$, we have $B\times B'\in \BB^{**}$. Observe
that every open set $U\subset T\times P$ is a countable union of
products of open sets. Write $U=\bigcup_{n=1}^{\infty}(U_n\times
V_n)$ with $U_n$ and $V_n$ open subsets of $T$ and $P$,
respectively. Then
\begin{multline*}
\widetilde{\pi}\left((B\times B')\cap
U\right)=\widetilde{\pi}\left((B\times B')\cap
\bigcup_{n=1}^{\infty}(U_n\times
V_n)\right)\\
=\widetilde{\pi}\left(\bigcup_{n=1}^{\infty}\left((B\cap
U_n)\times (B'\cap
V_n)\right)\right)
=\bigcup_{n=1}^{\infty}\widetilde{\pi}\left((B\cap U_n)\times
(B'\cap
V_n)\right)\\
=\bigcup_{n=1}^{\infty}\left(\pi(B\cap U_n)\times \pi(B'\cap
V_n)\right)\,,
\end{multline*}
which is in $\langle\BB_\LB\times\BB_\LB\rangle$.
\end{proof}

\begin{defn}\label{d:extension}
For $B\in \pi(\BB^*,\BB_\LB)$, let
$$\widetilde{\LB}(B)=\int_T\#\left(B\cap (\{t\}\times P)\right)\,d\LB(t)=\int_T\left(\int_{\{t\}\times P}\chi_{B\cap (\{t\}\times P)}\,d\nu)\right)\,d\LB(t)\;,$$
where $\nu$ denotes the counting measure and $\chi_X$ the
characteristic function of a subset $X\subset \{t\}\times P$.
\end{defn}
\begin{rem}
A measure on $T$ induces a transverse invariant measure on
$T\times P$. When $B$ is a generalized transversal,
$\widetilde{\LB}(B)$ is the value of this transverse invariant
measure on $B$. Therefore Definition \ref{d:extension} defines an
extension of this transverse invariant measure to a map
$\widetilde{\LB}: \pi(\BB^*,\BB_\LB)\to[0,\infty]$.
\end{rem}
\begin{prop}\label{p:biendef}
On $B\in \pi(\BB^*,\BB_\LB)$, $\widetilde{\LB}$ is well defined
and satisfies the following properties:
\begin{enumerate}
\item[\rm(a)]$\widetilde{\LB}(\emptyset)=0.$

\item[\rm(b)]$\widetilde{\LB}(\bigcup_{n\in\N}B_n)=\sum_{n=1}^{\infty}\widetilde{\LB}(B_n)$
for every countable family of disjoint sets
$B_n\in\pi(\BB^*,\BB_\LB), n\in\N.$
\end{enumerate}
\end{prop}
\begin{proof}
$\widetilde{\LB}$ is well defined if and only if the function
$h:T\to \R\cup\{\infty\}$, $h(t)=\#(B\cap \left(\{t\}\times
P)\right)$, is measurable with respect to the \sg-algebra
$\BB_\LB$ in $T$; {\em i.e.\/}, if $h^{-1}(\{n\})\in \BB_\LB$ for
all $n\in\N$. To prove this property, we proceed by induction on $n$. It is clear that
$h^{-1}(\{0\})=T\setminus \pi(B)$ belongs to $\BB_\LB$ since
$B\in\pi(\BB^*,\BB_\LB)$. Now, suppose $h^{-1}(\{i\})\in \BB_\LB$
for $i\in\{0,...,n-1\}$ and let us check that $h^{-1}(\{n\})\in
\BB_\LB$. Let $$C_n=\{\,((t,p_1),(t,p_2),...,(t,p_{n+1}))\ |\ t\in
T,\ p_1,...,p_{n+1}\in P\, \}\,,$$ which is a closed in $(T\times
P)^{n+1}$. Observe that $C_n$ is the set of $(n+1)$-uples in
$T\times P$ that lie in the same plaque. We remark that $C_n$ is
homeomorphic to $\Delta_T\times P^{n+1}$, where $\Delta_T$ is the
diagonal of the product $T^{n+1}$, and the projection
$\pi_T:\Delta_T\to T$, $(t,...,t)\mapsto t$ is a homeomorphism.
The measure $\LB$ becomes a measure on $\Delta_T$ via $\pi_T$. The
intersection $B^{n+1}\cap C_n$, denoted by $D_n$, is the set of
$(n+1)$-uples in $B$ that lie in the same plaque. Let
$$\Delta_n=\{\,((t,p_1),(t,p_2),...,(t,p_{n+1}))\in C_n\ |\ \exists\,
i, j\ \text{with}\ i\neq j\, \text{and}\, p_i=p_j\,\}\,,$$ which is closed in
$C_n$. This set consists of the $(n+1)$-uples in each plaque such
that two components are equal. The set $D_n\setminus\Delta_n$
consists of the $(n+1)$-uples of different elements in $B$ that
lie in the same plaque. Therefore,
$\pi_T\circ\pi_\Delta(D_n\setminus\Delta_n)$ consists of the
points $t\in T$ such that the corresponding plaque $\{t\}\times P$
contains more than $n$ points of $B$, where $\pi_\Delta:C_n\to
\Delta_T$ is the natural projection.

Now, let us prove that $\pi_\Delta(D_n\setminus\Delta_n)\in
\pi^{-1}_T(\BB_\LB)=\pi^{-1}_T(\BB)_\LB$. By Lemma
\ref{l:productos}, $\widehat{\pi}\left(B^{n+1}\setminus
\Delta_n\right)\in\langle\BB_\LB^{n+1}\rangle$, where
$\widehat{\pi}:(T\times P)^{n+1}\to T^{n+1}$ is the natural
projection. Therefore
$$\pi_\Delta(D_n\setminus\Delta_n)=\Delta_T\cap\widehat{\pi}\left(B^{n+1}\setminus
\Delta_n\right)\in \langle\BB_\LB^{n+1}\rangle|_{\Delta_T}\,,$$
where $\langle\BB_\LB^{n+1}\rangle|_{\Delta_T}$ denotes the
restriction of the \sg-algebra
$\langle\BB_\LB^{n+1}\rangle$ to $\Delta_T$. We only have to prove that $\langle
B_\LB^{n+1}\rangle|_{\Delta_T}\subset \pi^{-1}_T(\BB_\LB)$. For
that purpose, we have to check that the generators
$\prod_{k=1}^{n+1}F_k$, with $F_k\in\BB_\LB$, satisfy
$(\prod_{k=1}^{n+1}F_k)\cap\Delta_T\in \pi^{-1}_T(\BB_\LB)$. For
each $k$, take $A_k,B_k\in\BB$ with $A_k\subset F_k\subset B_k$
and $\LB(B_k\setminus A_k)=0$. Then
$$\left(\prod_{k=1}^{n+1}A_k\right)\cap\Delta_T\subset \left(\prod_{k=1}^{n+1}F_k\right)\cap\Delta_T\subset \left(\prod_{k=1}^{n+1}B_k\right)\cap\Delta_T\,,$$
and $\left(\prod_{k=1}^{n+1}A_k\right)\cap\Delta_T$ and
$\left(\prod_{k=1}^{n+1}B_k\right)\cap\Delta_T$ belong to
$\pi_T^{-1}(\BB)$ because
\begin{align*}
\left(\prod_{k=1}^{n+1}A_k\right)\cap\Delta_T=\bigcap_{k=1}^{n+1}\pi_T^{-1}(A_k)\;,\qquad
\left(\prod_{k=1}^{n+1}B_k\right)\cap\Delta_T=\bigcap_{k=1}^{n+1}\pi_T^{-1}(B_k)\,.
\end{align*}
Moreover
\begin{multline*}
\LB\left(\left(\left(\prod_{k=1}^{n+1}B_k\right)\cap\Delta_T\right)\setminus\left(\left(\prod_{k=1}^{n+1}A_k\right)\cap\Delta_T\right)\right)\\
=\LB\left(\bigcap_{k=1}^{n+1}(\pi_T^{-1}(B_k)\setminus \pi_T^{-1}(A_k))\right)
= \LB\left(\pi_T^{-1}\left(\bigcap_{k=1}^{n+1}(B_k\setminus A_k)\right)\right)\\
= \LB\left(\bigcap_{k=1}^{n+1}(B_k\setminus A_k)\right)
=0\;.
\end{multline*}
This shows that $\pi_\Delta(D_n\setminus\Delta_n)\in
\pi^{-1}_T(\BB_\LB)=\pi^{-1}_T(\BB)_\LB$. By induction, we have
$$h^{-1}(n)=T\setminus
\left((\pi_T\circ\pi_\Delta(D_n\setminus\Delta_n)\cup
h^{-1}(\{0,...,n-1\})\right)\in \BB_\LB.$$

Property (a) is obvious. To show property (b), observe that
$\chi_{\bigcup B_n}=\sum\chi_{B_n}$, and then use the monotonous
convergence theorem.
\end{proof}

\begin{defn}
If $B\in\BB^*\setminus\pi(\BB^*,\BB_\LB)$, then define
$\widetilde{\LB}(B)=\infty$.
\end{defn}
\begin{prop}\label{p:extension en productos}
$(T\times P,\ \BB^*,\ \widetilde{\LB})$ is a measure space and
$\widetilde{\LB}$ extends \LB.
\end{prop}
\begin{proof}
We only have to prove that
$\widetilde{\LB}(\bigcup_{n}B_n)=\sum_{n}\widetilde{\LB}(B_n)$
for every countable family of disjoint sets $B_n, n\in\N$, in
$\BB^*$. By Proposition \ref{p:biendef}, this holds if
$B_n\in\pi(\BB^*,\BB_\LB)$ for all $n\in\N$. If $\bigcup_n
B_n\in\BB^*\setminus\pi(\BB^*,\BB_\LB)$, then the above equality is obvious. So we
only have to consider the case where some
$B_j\in\BB^*\setminus\pi(\BB^*,\BB_\LB)$ and, however, $\bigcup_n
B_n\in\pi(\BB^*,\BB_\LB)$. We can suppose $B_j=B_1$, and let
$B=\bigcup_n B_n$. Let $$B^{\infty}=\{\,t\in T\ |\ \#(B\cap
(\{t\}\times P))=\infty\,\}\;,$$ which belongs to $\BB_\LB$ by
Proposition \ref{p:biendef}. The proof will be finished by
checking that $\LB(B^{\infty})>0$. We have $B^{\infty}_{1}\subset
B^{\infty}$, where
$$B_1^{\infty}=\{\,t\in T\ |\ \#(B_1\cap (\{t\}\times P))=\infty\,\}\;.$$
Suppose $\LB(B^{\infty})=0$. Since $\BB^{\infty}\in\BB_\LB$, there
is some $A\in \BB$ such that $B^{\infty}\subset A$ and $\LB(A)=0$.
The Borel set $\pi^{-1}(A)$ satisfies
$B_1\cap\pi^{-1}(A)\in\pi(\BB^*,\BB_\LB)$ since
$\emptyset\subset\pi(B_1\cap\pi^{-1}(A)\cap U)\subset A$ and
$\LB(A)=0$ for each open set $U\subset T\times P$. On the other
hand, $B_1\setminus\pi^{-1}(A)$ is a Borel set meeting every
plaque in a finite set, which is \sg-compact, and therefore
projects to a Borel set by Proposition \ref{p:sgcompactos}. Hence
 $B_1\in\pi(\BB^*,\BB_\LB)$ by Proposition \ref{p:propiedadunion},
which is a contradiction.
\end{proof}

We have constructed an extension of each transverse invariant
measure in a product foliated measurable space, but its uniqueness
was not proved. This uniqueness is false in general. For instance, take the
foliated product $\R\times\{\ast\}$ and let $\LB$ be the null
measure on the singleton $\{\ast\}$; our extension $\widetilde{\LB}$ is the zero
measure in the total space. Now, let $\mu$ be the measure defined
by
\begin{enumerate}
\item[(i)]  $\mu(B)=0$ for all countable set $B$; and

\item[(ii)] $\mu(B)=\infty$ for all uncountable Borel set $B$.

\end{enumerate}
This measure $\mu$ extends $\LB$ too and is quite different from
$\widetilde{\LB}$. In order to solve this problem, we require some
conditions to the extension. These conditions have the spirit of
coherency with the concept of transverse invariant measures. We
will prove that our extension is the unique coherent extension.

%--------------------------------------------------------------------

\begin{defn}
Let $\mu$ be an extension of a transverse invariant measure $\LB$
on $T\times P$. The measure $\mu$ is
called a {\em coherent extension\/} of $\LB$ if satisfies the
following conditions:
\begin{enumerate}
\item[(a)] If $B\in\BB^*$, $B\not\subset\pi^{-1}(S)$ for any $S\in\mathcal{B}$ with $\LB(S)=0$, and $\#B\cap
\{t\}\times P=\infty$ for each plaque $\{t\}\times P$ which meets $B$, then
$\mu(B)=\infty$.

\item[(b)]If $\LB(S)=0$ for some $S\in\mathcal{B}$, then
$\mu(\pi^{-1}(S))=0$.

\item[(c)] If $B\in\BB^*$ and $\LB(S)=\infty$ for all $S\in\mathcal{B}$ with $B\subset\pi^{-1}(S)$, then $\mu(B)=\infty$.
\end{enumerate}
\end{defn}

\begin{rem}
Condition (a) determines $\mu$ on Borel sets with infinite points
in plaques which are not contained in the saturation of a \LB-null
set. Condition (b) means certain coherency between the support of
$\LB$ and the support of the extension $\mu$. Condition (c)
determines $\mu$ on any Borel set so that any Borel set containing
its proyection has infinity \LB-measure.
\end{rem}

\begin{prop}\label{p:unicidad en productos}
$\widetilde{\LB}$ is the unique coherent extension.
\end{prop}
\begin{proof}
We prove that every coherent extension has the same values as
$\widetilde{\LB}$ on $\BB^*$. First, we consider the case
$B\in\pi(\BB^*,\BB_\LB)$. Let $$B^\infty=\{\,t\in T\ |\
\#(B\cap(\{t\}\times P))=\infty\,\}\;.$$ This set belongs to
$\BB_\LB$ by Proposition \ref{p:biendef}. Therefore there exist
Borel sets $A,C$ such that $A\subset B^\infty\subset C$ and
$\LB(C\setminus A)=0$. Let $\widetilde{B}^\infty=
B\cap\pi^{-1}(C)$. The Borel set $B\setminus\widetilde{B}^\infty$
is a generalized transversal and hence $\mu(B\setminus \widetilde{B}^{\infty})=\widetilde{\LB}(B\setminus \widetilde{B}^{\infty})$. On the other hand, if
$\LB(\pi(\widetilde{B}^\infty)))=0$, then $\LB(C)=0$ and
$\mu(\widetilde{B}^\infty)\leq\mu(\pi^{-1}(C))=0$ by (b). If
$\LB(B^\infty)) > 0$, let $\widehat{B}^{\infty}=B\cap\pi^{-1}(A)$
and $B'^{\infty}=B\cap\pi^{-1}(C\setminus A)$. Then
$\mu(\widehat{B}^\infty)=\infty$ by (a), and $\mu(B'^{\infty})=0$
by (b). Therefore $\mu$ equals $\widetilde{\LB}$ on
$\pi(\BB^*,\BB_\LB)$.

The case $B\in\BB^*\setminus\pi(\BB^*,\BB_\LB)$ is similar. The
set $B^{\infty}$ is not a Borel set in this case, but observe that
$(B\cap\pi^{-1}(B^\infty))\nsubseteq \pi^{-1}(S)$ with
$\LB(S)<\infty$ or we obtain $\pi(B\cap\pi^{-1}(S)\cap
U)\in\BB_\LB$ for all open set $U\subset T\times P$ by
Remark~\ref{r:analytic}, since $B\cap\pi^{-1}(S)\cap U$ is a Borel
set in $S\times P$ and \LB\ is finite in $S$. Hence
$B\in\pi(\BB^*,\BB_\LB)$ by Propositions~\ref{p:propiedadunion}
and \ref{p:sgcompactos}. Therefore $\mu(B)=\infty$ by (c). This
proves that $\mu$ and $\widetilde{\LB}$ agree on $\BB^*$, as
desired.
\end{proof}

%-----------------------------------------------------------------------

\section{The general case}

In this section, we prove the following theorem.

\begin{thm}\label{t:extension final}
Let $X$ be a foliated measurable space with a transverse invariant
measure \LB. There exists a measure $\widetilde{\LB}$ on $X$ that
extends $\LB$.
\end{thm}

Let $\{U_i,\varphi_i\}_{i\in\N}$ be a foliated measurable atlas
with $\varphi_i(U_i)=T_i\times\R^n$, where $T_i$ is a standard
Borel space. It is clear that $\varphi_i^{-1}( T_i\times\{\ast\})$
is a generalized transversal and, via $\varphi_i$, we obtain a
Borel measure $\LB_i$ on $T_i$. Proposition~\ref{p:extension en
productos} provides a measure $\widetilde{\LB}_i$ on
$U_i\thickapprox T_i\times\R^n$ that extends $\LB_i$. Moreover
Proposition \ref{p:sgcompactos} gives
$\widetilde{\LB}_i(T)=\LB(T)$ for all generalized transversal
$T\subset U_i$. Let $\pi_i:U_i\to
\varphi_i^{-1}(T_i\times\{\ast\})$ denote the natural projections.

We begin with a description of the change of foliated measurable
charts.

\begin{thm}[Kunugui, Novikov \cite{Srivastava}]\label{t:Srivastava}
Let $\{V_n\}_{n\in\N}$ be a countable base of open sets for a
Polish space $P$. Let $B\subset T\times P$ be a Borel set such
that $B\cap(\{t\}\times P)$ is an open set for all $t\in T$. Then
there exists a sequence $\{B_n\}_{n\in\N}$ of Borel subsets of $T$
such that $$B=\bigcup_n (B_n\times V_n).$$
\end{thm}

We take a countable base $\{V_m\}_{m\in\N}$ of $\R^n$ by connected
open sets.

\begin{lemma}\label{l:cociclo}
For $i,j\in\N$, there exists a sequence of Borel subsets of $T_i$,
$\{B_m\}_{m\in\N}$, and a sequence of open sets $\{W_m\}_{m\in\N}$ such that $\varphi_i(U_i\cap
U_j)=\bigcup_m(B_m\times W_m)$ and
$\varphi_j\circ\varphi_i^{-1}(t,x)= (f_{ijm}(t), g_{ijm}(t,x))$
for $(t, x)\in B_m\times W_m$, where each $f_{ijm}$ is a Borel
isomorphism.
\end{lemma}
\begin{proof}
We apply Theorem \ref{t:Srivastava} to $\varphi_j(U_i\cap U_j)$
and obtain a family $\{B'_m\}_{m\in\N}$ such that
$\varphi_j(U_i\cap U_j)=\bigcup_m(B'_m\times V_m)$. Now we apply
Theorem \ref{t:Srivastava} to each set
$\varphi_i\circ\varphi_j^{-1}(B'_k\times V_k)$, $k\in\N$. We
obtain sequences $B'_{k,n}$ such that
$\varphi_i\circ\varphi_j^{-1}(B'_k\times V_k)=
\bigcup_m(B'_{k,m}\times V_m)$, $k\in\N$. The set
$\varphi_j\circ\varphi_i^{-1}(\{t\}\times V_m)$, $t\in B'_{k,m}$,
is contained in only one plaque $\{\ast\}\times\R^n$ since each
$V_m$ is connected. Therefore, $\varphi_j\circ\varphi_i^{-1}(t,x)=
(f_{ijkm}(t), g_{ijkm}(t,x))$ for $(t,x)\in B'_{k,m}\times V_m$.
We must show that $f_{ijkm}$ is one-to-one. If there exist
$t,t'\in T_i$ with $f_{ijkm}(t)=f_{ijkm}(t')=t''$, then
$g_{ijkm}(\{t\}\times V_m)$ and $g_{ijkm}(\{t'\}\times V_m)$ are
connected open subsets of the plaque $\{t''\}\times V_k$ in
$B'_{k}\times V_k$, but this plaque is the image by
$\varphi_j\circ\varphi_i^{-1}$ of a connected open set since
$\varphi_j$ and $\varphi_i$ are homeomorphisms. Hence this set is
contained in a single plaque of $U_i$. This is a contradiction
with the assumption $t\neq t'$.

The functions $f_{ijkm}$ are, obviously, measurable and they have
measurable image by Proposition \ref{p:sgcompactos}, hence the image is a
standard Borel space. Since $f_{ijkn}$ is a one-to-one measurable
function between standard Borel spaces, it is a Borel isomorphism
\cite{Srivastava}. The proof is completed by observing that the required sequence $B_m\times W_m$ is the bisequence $B'_{k,m}\times V_m$ $k,m\in\N$.
\end{proof}

\begin{lemma}\label{l:coincidencia en proyectores}
Let $B$ be a Borel subset of $U_i\cap U_j$, $i,j\in \N$. Then
$$ B\in\pi_i(\BB^*,\BB_{\LB_i})\Longleftrightarrow
B\in\pi_j(\BB^*,\BB_{\LB_j})\;.$$
\end{lemma}
\begin{proof}
By Lemma \ref{l:cociclo}, there exists a countable family of
measurable holonomy transformations from
$\varphi_i^{-1}(T_i\times\{\ast\})$ to
$\varphi_j^{-1}(T_j\times\{\ast\})$ whose domains and ranges cover
$\pi_i(U_i\cap U_j)$ and $\pi_j(U_i\cap U_j)$, respectively.
Therefore, if $A$ is a Borel set contained in $U_i\cap U_j$ and
$\pi_i(A)$ is a Borel set, then $\pi_j(A)$ is a Borel set and
$$\LB(\pi_i(A))=0\Longleftrightarrow\LB(\pi_j(A))=0\;.$$
\end{proof}

\begin{lemma}\label{l:independencia de cartas}
$\widetilde{\LB}_i(B)=\widetilde{\LB}_j(B)$ for all Borel set
$B\subset U_i\cap U_j$, $i,j\in \N$.
\end{lemma}
\begin{proof}
We remark that $\widetilde{\LB}_i$ and $\widetilde{\LB}_j$ have
the same values in generalized transversals of $U_i\cap U_j$. By
Lemma \ref{l:coincidencia en proyectores}, we only consider Borel
sets in $\pi_i(\BB^*,\BB_\LB)$. Suppose that $\pi_i(B)$ is a Borel
set; otherwise, $\pi_i(B)$ is \LB-measurable and we can choose a
Borel set $A\subset \pi_i(B)$ with $\LB(\pi_i(B)\setminus A)=0$.
We take the Borel set $\widetilde{B}=B\cap\pi_i^{-1}(A)$. This
Borel set projects onto the Borel set $A$ and
$\widetilde{\LB}_i(B\setminus\widetilde{B})=\widetilde{\LB}_j(B\setminus\widetilde{B})=0$
by Definition~\ref{d:extension} and Lemma~\ref{l:coincidencia en
proyectores}, hence
$\widetilde{\LB}_i(B)=\widetilde{\LB}_i(\widetilde{B})$ and
$\widetilde{\LB}_j(B)=\widetilde{\LB}_j(\widetilde{B})$. Let
$$B^k=\{\,t\in T_i\ |\ \#(\varphi_i(B)\cap
(\{t\}\times\R^n))=k\,\}\;,\qquad k\in\N\cup\{\infty\}\;.$$ These are
\LB-measurable sets by Proposition \ref{p:biendef}, and we assume
that they are Borel sets by the same reason as above. Let
$\widetilde{B}^k$ denote
$B\cap\pi_i^{-1}(\varphi_i^{-1}(B^k\times\{\ast\}))$, which is a
Borel set. It is obvious that
$\bigcup_{i=1}^\infty\widetilde{B}^k$ is a generalized
transversal, hence
$\widetilde{\LB}_i(\bigcup_{k=1}^\infty\widetilde{B}^k)=\widetilde{\LB}_j(\bigcup_{k=1}^\infty\widetilde{B}^k)$.
Now consider
$$\widetilde{B}^\infty_l=\{\,x\in\widetilde{B}^\infty\ |\
\#(B\cap P^j_x)=l\,\}\;,\qquad l\in\N\cup\{\infty\}\;,$$ where $P^j_x$
denotes the plaque of $U_j$ that contains $x$. The proof is
finished in the case $\LB(\pi_i(\widetilde{B}^\infty_\infty))=0$
(we can restrict to the case of a generalized transversal). If
$\LB(\pi_i(\widetilde{B}^\infty_\infty))>0$, then
$\LB(\pi_j(\widetilde{B}^\infty_\infty))>0$. Therefore we
obviously obtain
$\widetilde{\LB}_i(B)=\infty=\widetilde{\LB}_j(B)$.
\end{proof}

\begin{defn}\label{d:medida final}
Let $B$ be a measurable set in $X$, and $$B_1 = B\cap U_1\;,\ B_k
= (B\cap U_k)\setminus (B_1\cup...\cup B_{k-1})\;,$$ for $k\geq
2$. Define
$$\widetilde{\LB}(B)=\sum_{i=1}^{\infty}\widetilde{\LB}_i(B_i)\,.$$
\end{defn}

By Lemma~\ref{l:independencia de cartas}, it is easy to prove that
Definition~\ref{d:medida final} does not depend neither on the
ordering of the charts nor on the choice of the countable foliated
measurable atlas. It is also easy to prove that $\widetilde{\LB}$
extends \LB\ since both of them have the same values on
generalized transversals contained in each chart and, hence, in
every generalized transversal. Theorem \ref{t:extension final} is
now established.

\begin{defn}
Let $\mu$ be an extension of a transverse invariant measure $\LB$
on a foliated measurable space $X$. The measure $\mu$ is called a
{\em coherent extension\/} of $\LB$ if it is a coherent extension
on each foliated measurable chart with the induced transverse
invariant measure.
\end{defn}

\begin{cor}
The extension $\widetilde{\LB}$ is the unique coherent extension
of \LB.
\end{cor}

Theorem \ref{t:extension medida} gives a new interpretation of
transverse invariant measures. It can be also used to introduce the
following version of the concept of transversal for foliated
measurable spaces with transverse invariant measures.

\begin{defn}
Let $X$ be a foliated measurable space with a transverse invariant
measure \LB. A Borel subset of $X$ with finite
$\widetilde{\LB}$-measure is called a {\em \LB-generalized
transversal\/}.
\end{defn}

\begin{rem}
In Section \ref{s:producto}, we have only used that the plaques
are Polish spaces. We can weaken the conditions of foliated
measurable spaces taking charts with the form $T\times P$, where
$P$ is any connected and locally connected Polish space. In this
way our result can be extended to other interesting cases like
{\em measurable graphs} \cite{Bermudez}.
\end{rem}

\end{document}